\documentclass[12pt]{amsart}
\usepackage{amsmath, amsthm, amsfonts, amscd, amssymb, eucal, latexsym, mathrsfs, appendix, parskip, tikz, verbatim}
\usepackage[numbers,sort&compress]{natbib}
\usepackage{color,hyperref}
\hypersetup{colorlinks,breaklinks,
            linkcolor=red,urlcolor=red,
            anchorcolor=red,citecolor=blue}

\theoremstyle{definition}
\numberwithin{equation}{section}

\newtheorem*{rep@theorem}{\rep@title}
\newcommand{\newreptheorem}[2]{%
\newenvironment{rep#1}[1]{%
 \def\rep@title{#2 \ref{##1}}%
 \begin{rep@theorem}}%
 {\end{rep@theorem}}}
\makeatother

\newreptheorem{theorem}{Theorem}

\newtheorem{theorem}{Theorem}[section]
\newtheorem{corollary}[theorem]{Corollary}
\newtheorem{lemma}[theorem]{Lemma}
\newtheorem{proposition}[theorem]{Proposition}

\newtheorem{definition}[theorem]{Definition}
\newtheorem{remark}[theorem]{Remark}

\newtheorem{examples}[theorem]{Examples}

\newtheorem*{conjecture*}{Conjecture}
\newtheorem{conjecture}{Conjecture}[section]

\allowdisplaybreaks

\newcommand{\fai}{\varphi}

\def\A{G}

\begin{document}
\title{The quantum group fixing  a sequence of finite subsets}
\author{Huichi Huang}
\address{Huichi Huang, College of Mathematics and Statistics, Chongqing University, Chongqing, 401331, PR China}
\email{huanghuichi@cqu.edu.cn}
\thanks{The author is partially supported by the Fundamental Research Funds for the Central Universities No. 0208005202045.}
\date{\today}
\begin{abstract}
Motivated by generalizing Szemer\'edi's theorem, we the elements in a discrete quantum group fixing a sequence of finite subsets and prove that the set of these elements  is a quantum subgroup. Using this we obtain a version of mean ergodic theorem for discrete quantum groups.
\end{abstract}

\maketitle
\section{Introduction}

In~\cite{Szemeredi1975}, E. Szemer\'edi proved the following theorem conjectured by P. Erd\"os and P. Tur\'an~\cite{ET1936}, which generalizes van der Waerden's theorem~\cite{vanderWaerden1927}.

\begin{theorem}~[Szemer\'edi's theorem]\

A set of positive integers with positive upper density contains arbitrarily long arithmetic progressions.
\end{theorem}

Szemer\'edi's original proof is combinatorial and has merits on its own right~\cite{Szemeredi1969,Szemeredi1975}.  A good survey is ~\cite{Tao2007}.



One may ask the following question:

{\it what's the reason behind the fact that a set with positive upper density contains arbitrarily long arithmetic progressions?}

In this paper, we prove a generalized Szemer\'edi's theorem and give a partial answer to this question.

\begin{theorem}~\label{APDG}
Let  $\Sigma=\{F_n\}$ be a sequence of finite subsets in a discrete group $\Gamma$ and suppose $b$ in $\Gamma$ fixes $\Sigma$ from right~(left). If a subset $\Lambda$ of $\Gamma$ has positive upper density with respect to $\Sigma$, then for any positive integer $k$, there exist $n>0$ and $a\in\Gamma$ such that  $\{b^{jn}a\}_{j=0}^{k-1}$~($\{ab^{jn}\}_{j=0}^{k-1}$) is contained in $\Lambda$ .
\end{theorem}

Here  we say that $b$ in $\Gamma$  {\bf fixes $\Sigma$} if $\displaystyle\lim_{n\to\infty}\frac{|b F_n\Delta F_n|}{|F_n|}=0$~($\displaystyle\lim_{n\to\infty}\frac{|F_n b\Delta F_n|}{|F_n|}=0$).

The {\bf upper density} $\displaystyle\overline{D}_\Sigma(\Lambda)$ of a subset
$\Lambda$ of $\Gamma$ with respect to  $\Sigma$ is defined by
$\displaystyle\limsup_{n\to\infty}\frac{|F_n\cap \Lambda|}{|F_n|}$~\cite{BBF2010}.

It's easy to see that the set $\Gamma_\Sigma$ of elements in $\Gamma$ fixing $\Sigma$ is a subgroup of $\Gamma$.

For discrete quantum groups, one can still define the subset fixing a sequence of finite sets. Moreover we can prove that it is a discrete quantum subgroup.

\begin{theorem}
~\label{thm: gp}
Given a sequence $\Sigma$ of finite subsets of $\widehat{\mathbb{G}}_\Sigma$ for a compact quantum group $\mathbb{G}$, the $C^*$-algebra  $C^*(\widehat{\mathbb{G}}_\Sigma)$ is a compact quantum group.
\end{theorem}

Mean ergodic theorem for amenable discrete quantum group already appears in~\cite{Huang2016-2}. Using the concept defined above,  we prove a mean ergodic theorem in the setting of arbitrary discrete quantum groups.~\footnote{A discrete quantum group is the dual of a compact quantum group, hence we state the result in terms of compact quantum groups.}

\begin{theorem}
~\label{thm: met}
  \begin{enumerate}
    \item[(i)] Suppose $T$ is a  limit  of $\{\frac{1}{|F_n|_w}\sum_{\alpha\in F_n}d_\alpha\pi(\chi(\alpha))\}_{n=1}^\infty$ in $B(H)$. Then $TP=PT=P$ where $P$ is the orthogonal projection from $H$ onto $H_\Sigma:=\{x\in H|\,\pi(\chi(\alpha))x=d_\alpha x\, {\rm for\,all}\,\alpha\in\cup F_n\}$;
    \item[(ii)] If $y$ in $H$ belongs to ${\rm Orb}(x,\widehat{\mathbb{G}}_\Sigma)$, then
    $$\lim_{n\to\infty}\frac{1}{|F_n|_w}\sum_{\alpha\in F_n}d_\alpha\pi(\chi(\alpha))(x-y)=0.$$
  \end{enumerate}
\end{theorem}

The article is organized as follows. In section 1, we collect some basic facts about compact quantum groups. In section 2,  we  prove Theorem~\ref{APDG}, which uses the concept of the subgroup fixing a sequence of finite subsets in a discrete group. In the remaining sections, we turn to quantum groups. In section 3, we prove  Theorem~\ref{thm: gp} which says that in a discrete quantum group, the subset fixing a sequence of finite subsets is a quantum subgroup.  Then in section 4, we define the orbit of a vector in a Hilbert space under an action of discrete quantum group and lay down some basic properties. Then we  prove a generalized mean ergodic Theorem~\ref{thm: met}.

\section{Preliminaries}\

\subsection{Conventions}\

Within this paper, we use $B(H,K)$ to denote the space of bounded linear operators from a Hilbert space $H$ to another Hilbert space $K$, and $B(H)$ stands for $B(H,H)$.

A net  $\{T_\lambda\}\subseteq B(H)$ converges to $T\in B(H)$ under strong operator topology~(SOT) if $T_\lambda x\to Tx$ for every $x\in H$, and  $\{T_\lambda\}$  converges to $T\in B(H)$ under weak operator topology~(WOT) if
$\langle T_\lambda x,y\rangle\to \langle Tx,y\rangle$ for all $x,y\in H$.

The notation $A \otimes B$ always means the minimal tensor product of two $C^*$-algebras $A$ and $B$.

For a state $\fai$ on a unital $C^*$-algebra $A$, we use $L^2(A,\fai)$ to denote the Hilbert space of GNS representation of $A$ with respect to $\fai$. The image of  an $a\in A$ in $L^2(A,\fai)$ is denoted by $\hat{a}$.

In this paper all $C^*$-algebras are assumed to be unital and separable.

\subsection{Some  Facts about Compact Quantum Groups}\

In this paper, we consider a discrete quantum group, which can be thought of as the dual of a compact quantum group. Compact quantum groups are  noncommutative analogues of compact groups~\cite{Woronowicz1987,BaajSkandalis1993,Woronowicz1998}.

\begin{definition}
A \textbf{compact quantum group} is a pair $(A,\Delta)$ consisting of a unital $C^*$-algebra $A$ and a unital $*$-homomorphism $$\Delta: A\rightarrow A\otimes A$$ such that
\begin{enumerate}
\item $(id\otimes\Delta)\Delta=(\Delta\otimes id)\Delta$.
\item $\Delta(A)(1\otimes A)$ and $\Delta(A)(A\otimes 1)$ are dense in $A\otimes A$.
\end{enumerate}
The $*$-homomorphism $\Delta$ is called the \textbf{coproduct} of $\A$.
\end{definition}

One may think of $A$ as $C(\mathbb{G})$, the $C^*$-algebra of continuous functions on a compact quantum space $\mathbb{G}$ with a quantum group structure. In the rest of the paper we write  a compact quantum group $(A,\Delta)$ as $\mathbb{G}$.

There exists a unique state $h$ on $A$ such that $$(h\otimes id)\Delta(a)=(id\otimes h)\Delta(a)=h(a)1_A$$
 for all $a$ in $A$. The state $h$ is called the \textbf{Haar measure} of $\mathbb{G}$. Throughout this paper, we use $h$ to denote it.

For a compact quantum group $\mathbb{G}$, there is a unique dense unital $*$-subalgebra $\mathcal{A}$ of $A$ such that
\begin{enumerate}
\item  $\Delta$  maps from $\mathcal{A}$ to $\mathcal{A}\odot \mathcal{A}$~(algebraic tensor product).
\item There exists a unique multiplicative linear functional $\varepsilon : \mathcal{A} \to \mathbb{C}$ and a linear map $\kappa : \mathcal{A} \to \mathcal{A}$ such that
$(\varepsilon \otimes id)\Delta(a) = (id \otimes \varepsilon) \Delta(a) = a$ and $m(\kappa \otimes id) \Delta(a) = m(id \otimes \kappa)\Delta(a) = \varepsilon(a)1$
for all $a \in \mathcal{A}$, where $m:\mathcal{A} \odot \mathcal{A}\to \mathcal{A}$ is the multiplication map. The functional $\varepsilon$ is called \textbf{counit} and $\kappa$ the \textbf{coinverse} of $C(\mathbb{G})$.
\end{enumerate}

Note that $\varepsilon$ is only densely defined and not necessarily bounded. If $\varepsilon$ is bounded and $h$ is faithful~($h(a^*a)=0$ implies $a=0$), then $\mathbb{G}$ is called \textbf{coamenable}~\cite{BMT2001}. Examples of coamenable compact quantum groups include $C(\mathbb{G})$ for a compact group $\mathbb{G}$ and $C^*(\Gamma)$ for a  discrete amenable group \,$\Gamma$.

A nondegenerate (unitary) \textbf{representation} $U$ of  a compact quantum group $\mathbb{G}$ is an invertible~(unitary) element in $M(K(H)\otimes A)$ for some Hilbert space $H$ satisfying that $U_{12}U_{13}=(id\otimes \Delta)U$. Here $K(H)$ is the $C^*$-algebra of compact operators on $H$ and  $M(K(H)\otimes A)$ is the multiplier $C^*$-algebra of  $K(H)\otimes A$.

We  write $U_{12}$ and $U_{13}$ respectively for the images of  $U$ by two maps from $M(K(H)\otimes A)$ to $M(K(H)\otimes A\otimes A)$ where the first one is obtained by extending the map $x \mapsto x \otimes 1$ from $K(H) \otimes A$ to $K(H) \otimes A\otimes A$, and the second one is obtained by composing this map with the flip on the  last two factors. The Hilbert space $H$ is called the \textbf{carrier Hilbert space} of $U$. From now on, we always assume representations are nondegenerate. If the carrier Hilbert space $H$ is of finite dimension, then $U$ is called a finite dimensional representation of $\mathbb{G}$.

For two representations $U_1$ and $U_2$ with the carrier Hilbert spaces $H_1$ and $H_2$ respectively, the set of
\textbf{intertwiners}  between $U_1$ and $U_2$, ${\rm Mor}(U_1,U_2)$, is defined by
$${\rm Mor}(U_1,U_2)=\{T\in B(H_1,H_2)|(T\otimes 1)U_1=U_2(T\otimes 1)\}.$$
Two representations $U_1$ and $U_2$ are equivalent if there exists a bijection $T$ in ${\rm Mor}(U_1,U_2)$.
A representation $U$ is called \textbf{irreducible} if ${\rm Mor}(U,U)\cong\mathbb{C}$.

Moreover, we have the following well-established facts about representations of compact quantum groups:
\begin{enumerate}
\item Every finite dimensional representation is equivalent to a unitary representation.
\item Every irreducible representation is  finite dimensional.
\end{enumerate}
Let $\widehat{\mathbb{G}}$ be the set of equivalence classes of irreducible unitary representations of $\mathbb{G}$. For every $\gamma\in \widehat{\mathbb{G}}$, let $U^{\gamma}\in \gamma$  be unitary and $H_{\gamma}$ be its carrier Hilbert space with dimension $d_{\gamma}$. After fixing an orthonormal basis of $H_{\gamma}$, we can write $U^{\gamma}$ as  $(u^{\gamma}_{ij})_{1\leq i,j\leq d_{\gamma}}$ with $u^{\gamma}_{ij}\in A$~($u^{\gamma}_{ij}$'s are called the \textbf{matrix entries} of $\gamma$), and
$$\Delta(u^{\gamma}_{ij})=\sum_{k=1}^{d_\gamma}u^{\gamma}_{ik}\otimes u^{\gamma}_{kj}$$ for all $1\leq i,j\leq d_\gamma$.

The matrix $\overline{U^{\gamma}}$ is still an irreducible representation~(not necessarily unitary) with the carrier Hilbert space $\bar{H}_\gamma$. It is called the \textbf{conjugate} representation of $U^\gamma$ and the equivalence class of $\overline{U^{\gamma}}$ is denoted by $\bar{\gamma}$.

Given two finite dimensional representations  $\alpha,\beta$ of $\mathbb{G}$,  fix orthonormal basises for $\alpha$ and $\beta$ and write $\alpha,\beta$ as $U^\alpha, U^\beta$ in matrix forms respectively. Define the \textbf{ direct sum}, denoted by $\alpha+\beta$ as an equivalence class of unitary representations of dimension $d_\alpha+d_\beta$ given by
$\bigl(\begin{smallmatrix}
U^\alpha&0\\ 0&U^\beta
\end{smallmatrix} \bigr)$, and
the \textbf{tensor product}, denoted by $\alpha\beta$,  is an equivalence class of unitary representations of dimension $d_\alpha d_\beta$ whose matrix form is given by $U^{\alpha\beta}=U^\alpha_{13}U^\beta_{23}$.

The \textbf{character} $\chi(\alpha)$ of a finite dimensional representation $\alpha$ is given by
$$\chi(\alpha)=\sum_{i=1}^{d_\alpha} u^\alpha_{ii}.$$ Note that $\chi(\alpha)$ is independent of choices of representatives of $\alpha$. Also $\|\chi(\alpha)\|\leq d_\alpha$ since $\sum_{k=1}^{d_\alpha} u^\alpha_{ik}(u^\alpha_{ik})^*=1$ for every $1\leq i\leq d_\alpha$. Moreover
$$\chi(\alpha+\beta)=\chi(\alpha)+\chi(\beta),\,\chi(\alpha\beta)=\chi(\alpha)\chi(\beta)\,{\rm and}\,\chi(\alpha)^*=\chi(\bar{\alpha})$$ for finite dimensional representations $\alpha,\beta$.

Every representation of a compact quantum group is a direct sum of irreducible representations. For two finite dimensional representations $\alpha$ and $\beta$, denote the number of copies of $\gamma\in\widehat{\mathbb{G}}$ in the  decomposition of $\alpha\beta$ into sum of irreducible representations by $N_{\alpha,\beta}^\gamma$. Hence
$$\alpha\beta=\sum_{\gamma\in \widehat{\mathbb{G}}}N_{\alpha,\beta}^\gamma\gamma.$$

We have the Frobenius reciprocity law~\cite[Proposition 3.4.]{Woronowicz1987}~\cite[Example 2.3]{Kyed2008}.
$$N_{\alpha,\beta}^\gamma=N_{\gamma,\bar{\beta}}^\alpha=N_{\bar{\alpha},\gamma}^\beta,$$ for all $\alpha,\beta,\gamma\in \widehat{\mathbb{G}}$.

Within the paper, we assume that $A=C(\mathbb{G})$ is a separable $C^*$-algebra, which amounts to say, $\widehat{\mathbb{G}}$ is countable.

\begin{definition}~\cite[Definition 3.2]{Kyed2008}~\label{boundary}\
Given two finite subsets $S, F$ of $\widehat{\mathbb{G}}$, the \textbf{boundary} of $F$ relative to $S$, denoted by $\partial_S(F)$, is defined by
\begin{align*}
\partial_S(F)=&\{\alpha\in F\,|\, \exists\,\gamma\in S,\, \beta\notin F,\,{\rm such\, that}\, N_{\alpha,\gamma}^\beta>0\,\}   \\
&\cup\{\alpha\notin F\,|\, \exists\,\gamma\in S, \,\beta\in F,\,{\rm such\, that}\, N_{\alpha,\gamma}^\beta>0\,\}.
\end{align*}
We denote  $\partial_{\{\gamma\in\widehat{\mathbb{G}}|\gamma \,{\rm is\, contained \, in\, \alpha}\}}(F)$ by $\partial_{\alpha}(F)$ for a finite dimensional representation $\alpha$.
\end{definition}

We say $\gamma$ in $\widehat{\mathbb{G}}$  \textbf{fixes  a sequence} $\Sigma=\{F_n\}$ of finite subsets in $\widehat{\mathbb{G}}$ if
$$\lim_{n\to\infty}\frac{|\partial_\gamma(F_n)|_w}{|F_n|_w}=0.$$

The \textbf{weighted cardinality} $|F|_w$ of a finite subset $F$ of $\widehat{\mathbb{G}}$ is given by
$$|F|_w=\sum_{\alpha\in F} d_\alpha^2.$$

\section{The Case for Groups}\

Let $G$ be a countable discrete group and $\Sigma=\{F_n\}_{n=1}^\infty$ is a sequence of finite subsets of $G$.
\begin{definition}
We say that an element $g$ in $G$  \textbf{fixes  $\Sigma$} if
  $$\lim_{n\to\infty}\frac{|gF_n\Delta F_n|}{|F_n|}=0.$$ Denote by $G_\Sigma$ the set of elements in $G$ fixed by $\Sigma$.
\end{definition}

It's routine to check that $G_\Sigma$ is a subgroup of $G$.

\begin{examples}~[Examples of $G_\Sigma$]\

  \begin{enumerate}
    \item A group $G$ is amenable iff $G_\Sigma=G$ and $\Sigma$ is a F\o lner sequence.
    \item In $\mathbb{Z}$, let $\Sigma=\{F_n\}_{n=1}^\infty$ with $F_n=\{1, 3,\cdots, 2n+1\}$. Then $\mathbb{Z}_\Sigma=2\mathbb{Z}$.
  \end{enumerate}
\end{examples}

\subsection{Arithmetic Progressions in Discrete Groups}\

In 1977, H. Furstenberg found that Szemer\'edi's theorem is equivalent to a multiple recurrence theorem, which he called ``ergodic Szemer\'edi theorem''.  See~\cite[Thm. 1.4]{Furstenberg1977} and ~\cite[Thm. II]{FKO1982}.

Via proving his ergodic Szemer\'edi theorem,   Furstenberg  gave an ergodic theoretic proof of Szemer\'edi's theorem. This is Furstenberg correspondence principle which opens a door for applications of ergodic theory to combinatorial number theory. This is also the main ingredient of the paper.

\begin{theorem}~[Ergodic Szemer\'edi theorem]\
~\label{thm: CP}

Let $(X,\mathcal{B},\nu, T)$ be a dynamical system consisting of a probability measure space $(X,\mathcal{B},\nu)$ and a measure-preserving transformation $T:X\to X$. For any positive integer $k$, there exists $n\in\mathbb{Z}$ such that $$\mu(\bigcap_{j=1}^k T^{-jn}A)>0$$ whenever $\mu(A)>0$.
\end{theorem}
Along this idea, it appear various generalizations  of Szemer\'edi's theorem to $\mathbb{Z}^d$~\cite{FK1978,FK1991,BL1996}.

Actually  via ergodic Szemer\'edi theorem, Furstenberg had proved a theorem stronger than  Szemer\'edi's theorem~\cite[Thm. I]{FKO1982}.
\begin{theorem}~[Furstenberg's version of Szemer\'edi's theorem]\

A set of positive integers with positive upper Banach density contains arbitrarily long arithmetic progressions.
\end{theorem}

A subset $\Lambda$ of positive integers has {\bf positive upper density} if $\displaystyle\limsup_{n\to\infty} \frac{|\Lambda\cap [1,n]|}{n}>0$ and has {\bf positive upper Banach density} if $\displaystyle\limsup_{n\to\infty} \frac{|\Lambda\cap [a_n, b_n)|}{b_n-a_n}>0$ for a sequence of intervals $\{[a_n, b_n)\}$ with $b_n-a_n\to\infty$.

Suppose $T$ is a homeomorphism on a compact metrizable space $X$. A Borel probability measure $\nu$ on $X$ is called {\bf $T$-invariant} if
$\nu(T^{-1}A)=\nu(A)$ for every Borel subset $A$ of $X$ and every $s$ in $\Gamma$.

It's well-known that $\nu$ is $T$-invariant if and only if $\nu(T^{-1}f)=\nu(f)$ for every $f$ in $C(X)$. Here $C(X)$ stands for the set of complex-valued continuous functions on $X$, $\nu(f)=\int_X f(y)\,d\nu(y)$ and $T^{-1} f(x)=f(T(x))$ for every $x$ in $X$.

Now we start to prove the first main Theorem, Theorem~\ref{APDG}, which relies on the multiple recurrence theorem due to Furstenberg. See~\cite[Thm. 1.4]{Furstenberg1977} and ~\cite[Thm. II]{FKO1982}.

\begin{theorem}~[Ergodic Szemer\'edi theorem]\
~\label{thm: CP}

Let $(X,\mathcal{B},\nu, T)$ be a dynamical system consisting of a probability measure space $(X,\mathcal{B},\nu)$ and a measure-preserving transformation $T:X\to X$. For any positive integer $k$, there exists $n\in\mathbb{Z}$ such that $$\mu(\bigcap_{j=1}^k T^{-jn}A)>0$$ whenever $\mu(A)>0$.
\end{theorem}

\begin{theorem}~\label{APDG}
Let  $\Sigma=\{F_n\}$ be a sequence of finite subsets in a discrete group $\Gamma$ and suppose $b$ in $\Gamma$ fixes $\Sigma$ from right~(left). If a subset $\Lambda$ of $\Gamma$ has positive upper density with respect to $\Sigma$, then for any positive integer $k$, there exist $n>0$ and $a\in\Gamma$ such that  $\{b^{jn}a\}_{j=0}^{k-1}$~($\{ab^{jn}\}_{j=0}^{k-1}$) is contained in $\Lambda$ .
\end{theorem}

\begin{proof}
We only need to give a proof for the case that $b$ in $\Gamma$ fixes $\Sigma$ from right.

If $b$ in $\Gamma$ fixes $\Sigma$ from left and $\Lambda$ is a subset of $\Gamma$ with positive upper density with respect to  $\Sigma=\{F_n\}$, then $b^{-1}$  fixes $\Sigma^{-1}=\{F_n^{-1}\}$ from left and $\Lambda^{-1}$ is a subset with positive upper density with respect to $\Sigma^{-1}$ .  The existence of $\{b^{jn}a\}_{j=1}^k$ in $\Lambda^{-1}$ gives the existence of $\{ab^{jn}\}_{j=1}^k$ in $\Lambda$.

Let $\Gamma$ act on $\{0,1\}^\Gamma$ by shift, that is, $s\cdot x(t)=x(ts)$ for all $s,t$ in $\Gamma$ and $x$ in $\{0,1\}^\Gamma$. Define $A_0:=\{x\in \{0,1\}^\Gamma| x(e_\Gamma)=1\}$ and let
$\omega=1_{\Lambda}$ be the characteristic function of $\Lambda$.

Suppose $b$ in $\Gamma$ fixes $\Sigma$ from left.

Then for every positive integer $k$, one has
\begin{align*}
\exists \,a\in\Gamma \, {\rm such \, that}\, \{b^ja\}_{j=1}^k\subseteq \Lambda;   &\Longleftrightarrow   \exists \,a \in\Gamma \, {\rm such \, that}\,\omega(b^ja)=1 \quad {\rm for \,all}  \, 1\leq j\leq k;   \\
& \Longleftrightarrow\exists\, a\in\Gamma \, {\rm such \, that}\, b^ja\cdot \omega(e_\Gamma)=1\quad {\rm for\, all}  \, 1\leq j\leq k;  \\
&\Longleftrightarrow\exists \,a\in\Gamma \, {\rm such \, that}\, \{b^ja\cdot \omega\}_{j=1}^k\subseteq A_0.
\end{align*}

Let $X=\overline{\Gamma\cdot\omega}$ be the closure of the orbit of $\omega$ in $\{0,1\}^\Gamma$.

Let $A=A_0\cap X$, which is a closed subset of $X$.

It follows that
\begin{align*}
&\exists \,a\in\Gamma \, {\rm such \, that}\,\displaystyle\{b^ja\cdot \omega\}_{j=1}^k\subseteq A_0;\Longleftrightarrow \exists \,a\in\Gamma \, {\rm such \, that}\,\omega\in\bigcap_{j=1}^k (b^ja)^{-1}\cdot A_0; \\
&\Longleftrightarrow\exists \,a\in\Gamma \, {\rm such \, that}\, a\cdot\omega\in\bigcap_{j=1}^k b^{-j}\cdot A_0; \Longleftrightarrow\bigcap_{j=1}^k b^{-j}\cdot A_0\cap\Gamma\cdot\omega\neq\emptyset; \\
\tag{$A_0$ is open$\Longrightarrow \bigcap_{j=1}^k b^{-j}\cdot A_0$ is open.}\\
&\Longleftrightarrow\bigcap_{j=1}^k b^{-j}\cdot A_0\cap\overline{\Gamma\cdot\omega}\neq\emptyset;
 \Longleftrightarrow \bigcap_{j=1}^k b^{-j}A\neq\emptyset.
\end{align*}

Next we are going to construct a $b$-invariant Borel probability measure $\mu$ on $X$ such that $\displaystyle\mu(A)>0$.  By Theorem~\ref{thm: CP}, this will complete the proof.

Let $\delta_{s\cdot\omega}$ be the Dirac measure at the point $s\cdot\omega$ for $s$ in $\Gamma$,  and this is a Borel probability measure on $X$.

Define $\mu_n=\frac{1}{|F_n|}\sum_{t\in F_n} \delta_{t\cdot\omega}$.  Let $\mu$ be a weak-$*$ limit of $\mu_n$. Without loss of generality, let $\displaystyle\mu=\displaystyle\lim_{n\to\infty}\mu_n$.

Then the following two claims hold.
\begin{enumerate}
\item $\mu$ is $b$-invariant.
\item  $\mu(A)>0$.
\end{enumerate}
\begin{proof}~[Proof of the first claim]

For every continuous function $f$ on $X$, one has
\begin{align*}
\mu(b^{-1}\cdot f)=&\lim_{n\to\infty}\mu_n(b^{-1}\cdot f)=\lim_{n\to\infty}\frac{1}{|F_n|}\sum_{t\in F_n}b^{-1}\cdot f(t\cdot\omega)  \\
=&\lim_{n\to\infty}\frac{1}{|F_n|}\sum_{t\in F_n} f(bt\cdot\omega)=\lim_{n\to\infty}\frac{1}{|F_n|}\sum_{t\in bF_n}f(t\cdot\omega)  \\
\tag{$b$ fixes $\Sigma$ from right.} \\
=&\lim_{n\to\infty}\frac{1}{|F_n|}\sum_{t\in F_n}f(t\cdot\omega)=\mu(f).
\end{align*}
Hence $\mu$ is $b$-invariant.
\end{proof}

\begin{proof}~[Proof of the second claim]

Note that
\begin{align*}
\limsup_{n\to\infty}\mu_n(A)&=\limsup_{n\to\infty}\frac{|\{t\in F_n|\,t\cdot\omega\in A\}|}{|F_n|}  \\
&=\limsup_{n\to\infty}\frac{|\{t\in F_n|\,t\cdot\omega\in A_0\}|}{|F_n|}  \\
&=\limsup_{n\to\infty}\frac{|\{t\in F_n|\,t\cdot\omega(e_\Gamma)=1\}|}{|F_n|}  \\
&=\limsup_{n\to\infty}\frac{|\{t\in F_n|\,\omega(t)=1\}|}{|F_n|}  \\
&=\limsup_{n\to\infty}\frac{| F_n\cap\Lambda |}{|F_n|}  \\
&=\overline{D}_\Sigma(\Lambda)>0.
\end{align*}

Since $A$ is a closed subset of $X$, we have $\displaystyle\mu(A)\geq\displaystyle\limsup_{n\to\infty}\mu_n(A)>0$~\cite[Sec. 6.1, Remarks (3)]{Walters1982}.
\end{proof}
Applying Theorem~\ref{thm: CP} to the dynamical system $(X,\mu,b)$ gives the proof.

\end{proof}

\begin{remark}
A set $\Lambda$ has positive upper density with respect to  $\Sigma$ iff it has positive density with respect to a subsequence of $\Sigma$, hence without loss of generality, we can just assume that $\Lambda$ has positive  density with respect to a  sequence.  We implicitly use this fact in the proof of Theorem~\ref{APDG}.
\end{remark}

A sequence $\{F_n\}_{n=1}^\infty$ of finite subsets in a countable discrete group $\Gamma$ is called a left~(right) {\bf F\o lner sequence} if
$$\displaystyle\lim_{n\to\infty}\frac{|sF_n\Delta F_n|}{|F_n|}=0   \quad
(\displaystyle\lim_{n\to\infty}\frac{|F_ns\Delta F_n|}{|F_n|}=0)$$ for every $s$ in $\Gamma$.  A group $\Gamma$ having  a F\o lner sequence is called {\bf amenable}.

\begin{remark}
It might happen that  except the neutral element,  no other element in $\Gamma$ fixes $\Sigma$ for a sequence $\Sigma$. For instance no integer except 0 fixes $\Sigma=\{F_n\}_{n=1}^\infty$ for $F_n=\{2, 4,\cdots, 2^n\}$.  So choices of $\Sigma$ decide the elements fixed by it.

When $\Gamma$ is amenable, and one can choose $\Sigma$ to be a left~(right) F\o lner sequence of $\Gamma$. Then every $b$ in $\Gamma$ fixes $\Sigma$ from right~(left).
\end{remark}
So Theorem~\ref{APDG} gives  the following application.

\begin{corollary}~[Arithmetic progressions in amenable groups]
~\label{cor: SZM}\

In a subset $\Lambda$ of an amenable group $\Gamma$ with positive upper density with respect to a left~(right) F\o lner sequence, for every positive integer $k$ and every $b$ in $\Gamma$, there exist $a$ in $\Gamma$ and a positive integer $n$ such that $\{b^{jn}a\}_{j=1}^k$~($\{ab^{jn}\}_{j=1}^k$) is contained in $\Lambda$.

Moreover if $\Gamma$ contains $\mathbb{Z}$ as a subgroup, then a subset of $\Gamma$ with positive upper density with respect to a  left~(right) F\o lner sequence contains arbitrarily long left~(right) arithmetic progressions.

\end{corollary}
\begin{proof}
If $\Sigma$ is a left F\o lner sequence in an amenable group $\Gamma$, then every $b$ in $\Gamma$ fixes $\Sigma$ from right.  By Theorem~\ref{APDG}, the first statement holds.

Since $\mathbb{Z}$ is a subgroup of $\Gamma$, there exists an element b of infinite order in $\Gamma$.  By Theorem~\ref{APDG},  for every positive integer $k$, there exist $a$ in $\Gamma$ and a positive integer $n$ such that $\Lambda$ contains $\{b^{jn}a\}_{j=1}^k$~($\{ab^{jn}\}_{j=1}^k$).  Since $b$ is of infinite order, the set $\{b^{jn}a\}_{j=1}^k$~($\{ab^{jn}\}_{j=1}^k$) has $k$ distinct elements.  Hence it is a left~(right) arithmetic progression of length $k$.
\end{proof}

\section{The Case for Quantum Groups}


\subsection{The Quantum Subgroup Fixing A Sequence of Finite Subsets}\

Denote the elements in $\widehat{\mathbb{G}}$ fixed by a sequence of finite subsets $\Sigma=\{F_n\}_{n=1}^\infty$ by $\widehat{\mathbb{G}}_\Sigma$.

In this subsection we are going to prove that $\widehat{\mathbb{G}}_\Sigma$ is a quantum subgroup of $\widehat{\mathbb{G}}$. This amounts to say that the $C^*$-subalgebra $C^*(\widehat{\mathbb{G}}_\Sigma)$ generated by the matrix entries of all elements in $\widehat{\mathbb{G}}_\Sigma$ is a compact quotient group of $\mathbb{G}$, that is,  $C^*(\widehat{\mathbb{G}}_\Sigma)$ is a compact quantum group.
\begin{theorem}
~\label{thm: gp}
Given a sequence $\Sigma$ of finite subsets of $\widehat{\mathbb{G}}_\Sigma$ for a compact quantum group $\mathbb{G}$, the $C^*$-algebra  $C^*(\widehat{\mathbb{G}}_\Sigma)$ is a compact quantum group.
\end{theorem}
\begin{proof}
  We are going to verify the following:
  \begin{enumerate}
    \item The trivial representation $\gamma_0$ is in $\widehat{\mathbb{G}}_\Sigma$.
    \item If $\gamma$ is in $\widehat{\mathbb{G}}_\Sigma$, then its conjugate $\bar{\gamma}$ is also in $\widehat{\mathbb{G}}_\Sigma$.
    \item  If $\gamma_1,\gamma_2$ are in $\widehat{\mathbb{G}}_\Sigma$, then every $\gamma\in\widehat{\mathbb{G}}$ contained in $\gamma_1\gamma_2$ is also in $\widehat{\mathbb{G}}_\Sigma$.
  \end{enumerate}

Firstly  for every finite subset $F$ of $\widehat{\mathbb{G}}_\Sigma$, the set $\partial_{\gamma_0}(F)$ is always empty. So $\gamma_0$ is in $\widehat{\mathbb{G}}_\Sigma$.

To prove (2) and (3), we need a lemma.
\begin{lemma}~\label{lm: multiplicity}
For $\alpha, \beta,\gamma$ in $\widehat{\mathbb{G}}_\Sigma$, we have $\displaystyle\sum_{\beta\in\widehat{\mathbb{G}}_\Sigma,\, N_{\alpha,\gamma}^\beta>0}d_\beta^2\leq d_\alpha^2 d_\gamma^2$
 and $\displaystyle\sum_{\alpha\in\widehat{\mathbb{G}}_\Sigma,\, N_{\alpha,\gamma}^\beta>0}d_\alpha^2\leq d_\beta^2 d_\gamma^2$.
\end{lemma}
\begin{proof}
Note that
$$d_\alpha d_\gamma=\sum_{\beta\in\widehat{\mathbb{G}}} N_{\alpha,\gamma}^\beta d_\beta
 =\sum_{\beta\in\widehat{\mathbb{G}},\, N_{\alpha,\gamma}^\beta>0} N_{\alpha,\gamma}^\beta d_\beta
\geq \sum_{\beta\in\widehat{\mathbb{G}},\, N_{\alpha,\gamma}^\beta>0}  d_\beta.$$

So $d_\alpha^2 d_\gamma^2\geq (\sum_{\beta\in\widehat{\mathbb{G}},\, N_{\alpha,\gamma}^\beta>0}  d_\beta)^2\geq \sum_{\beta\in\widehat{\mathbb{G}}_\Sigma,\, N_{\alpha,\gamma}^\beta>0}d_\beta^2$.


Furthermore  by Frobenius reciprocity law, we have  $N_{\alpha,\gamma}^\beta=N_{\beta,\bar{\gamma}}^\alpha$. Hence
\begin{align*}
&d_\beta d_\gamma=d_\beta d_{\bar{\gamma}}=\sum_{\alpha\in\widehat{\mathbb{G}}} N_{\beta,\bar{\gamma}}^\alpha d_\alpha
 =\sum_{\beta\in\widehat{\mathbb{G}},\, N_{\beta,\bar{\gamma}}^\alpha>0} N_{\beta,\bar{\gamma}}^\alpha d_\alpha  \\
&=\sum_{\beta\in\widehat{\mathbb{G}},\, N_{\alpha,\gamma}^\beta>0} N_{\alpha,\gamma}^\beta d_\alpha  \geq\sum_{\alpha\in\widehat{\mathbb{G}},\, N_{\alpha,\gamma}^\beta>0}  d_\alpha.
\end{align*}
So $d_\beta^2 d_\gamma^2\geq (\sum_{\alpha\in\widehat{\mathbb{G}},\, N_{\alpha,\gamma}^\beta>0}  d_\alpha)^2\geq \sum_{\alpha\in\widehat{\mathbb{G}}_\Sigma,\, N_{\alpha,\gamma}^\beta>0}d_\alpha^2$.
\end{proof}

We prove (2) via proving that $\displaystyle\lim_{n\to\infty}\frac{|\partial_{\bar{\gamma}}(F_n)|_w}{|F_n|_w}=0$ provided $\displaystyle\lim_{n\to\infty}\frac{|\partial_\gamma(F_n)|_w}{|F_n|_w}=0$.

By definition
\begin{align*}
\partial_{\bar{\gamma}}(F_n)&=\{\alpha\in F_n\,| N_{\alpha,\bar{\gamma}}^\beta>0\, {\rm for\, some\,}\beta\notin F_n\}\cup
\{\alpha\notin F_n\,| N_{\alpha,\bar{\gamma}}^\beta>0\, {\rm for\, some\,}\beta\in F_n\}    \\
&\notag{(Frobenius \,reciprocity\, law)} \\
&=\{\alpha\in F_n\,| N_{\beta,\gamma}^\alpha>0\, {\rm for\, some\,}\beta\notin F_n\}\cup \{\alpha\notin F_n\,| N_{\beta,\gamma}^\alpha>0\, {\rm for\, some\,}\beta\in F_n\}.
\end{align*}
On the other hand
$$\partial_\gamma(F_n)=\{\alpha\in F_n\,| N_{\alpha,\gamma}^\beta>0\, {\rm for\, some\,}\beta\notin F_n\}\cup
\{\alpha\notin F_n\,| N_{\alpha,\gamma}^\beta>0\, {\rm for\, some\,}\beta\in F_n\}.$$

We can define a map $\fai$ from $\partial_{\bar{\gamma}}(F_n)$ to $\partial_\gamma(F_n)$ by the following:

when   $\alpha\in\partial_{\bar{\gamma}}(F_n)\cap F_n$, the image $\fai(\alpha)$ is given by some $\beta$  in $F_n^c$ with $N_{\alpha,\bar{\gamma}}^\beta>0$; when $\alpha\in\partial_{\bar{\gamma}}(F_n)\cap F_n^c$, the image $\fai(\alpha)$ is given by some $\beta$ in $F_n$ with $N_{\alpha,\bar{\gamma}}^\beta>0$. From $N_{\alpha,\bar{\gamma}}^\beta=N_{\beta,\gamma}^\alpha$, we have that $\fai(\alpha)$ is in $\partial_\gamma(F_n)$. By Lemma~\ref{lm: multiplicity}, we have
$\displaystyle\sum_{\fai(\alpha)=\beta}d_\alpha^2\leq d_\beta^2d_\gamma^2$.

Hence
$$|\{\alpha\in F_n\,| N_{\beta,\gamma}^\alpha>0\, {\rm for\, some\,}\beta\notin F_n\}|_w\leq d_\gamma^2 |\{\beta\notin F_n\,| N_{\alpha,\gamma}^\beta>0\, {\rm for\, some\,}\alpha\in F_n\}|_w$$ and
$$|\{\alpha\notin F_n\,| N_{\beta,\gamma}^\alpha>0\, {\rm for\, some\,}\beta\in F_n\}|_w\leq d_\gamma^2 |\{\beta\in F_n\,| N_{\alpha,\gamma}^\beta>0\, {\rm for\, some\,}\alpha\notin F_n\}|_w.$$

Therefore $|\partial_{\bar{\gamma}}(F_n)|_w\leq d_\gamma^2 |\partial_\gamma(F_n)|_w$ and (2) follows immediately.

Now we proceed to the proof of (3).

Suppose $\gamma_1$ and $\gamma_2$ are in $\widehat{\mathbb{G}}_\Sigma$. We are going to prove that
$$\lim_{n\to\infty}\frac{|\partial_{\gamma_1\gamma_2} (F_n)|_w}{|F_n|_w}=0.$$
By definition
$$\partial_{\gamma_1\gamma_2} (F_n)=\{\alpha\in F_n\,|N_{\alpha,\gamma_1\gamma_2}^\beta>0\, {\rm\, for\, some\,}\beta\notin F_n\}\cup\{\alpha\notin F_n \,|N_{\alpha,\gamma_1\gamma_2}^\beta>0\, {\rm\, for\, some\,}\beta\in F_n\}.$$
Since $\alpha(\gamma_1\gamma_2)=(\alpha\gamma_1)\gamma_2$, we have $N_{\alpha,\gamma_1\gamma_2}^\beta=N_{\alpha\gamma_1,\gamma_2}^\beta$ for all $\alpha,\beta$ in $\widehat{\mathbb{G}}$.

Hence
$$\{\alpha\in F_n\,|N_{\alpha,\gamma_1\gamma_2}^\beta>0\, {\rm\, for\, some\,}\beta\notin F_n\}=\{\alpha\in F_n\,|N_{\alpha\gamma_1,\gamma_2}^\beta>0\, {\rm\, for\, some\,}\beta\notin F_n\}.$$
Note that $N_{\alpha\gamma_1,\gamma_2}^\beta=\sum_{\gamma\in \widehat{\mathbb{G}}}N_{\alpha,\gamma_1}^\gamma N_{\gamma,\gamma_2}^\beta$.

Suppose $\alpha$ is in $F_n$ with $N_{\alpha,\gamma_1\gamma_2}^\beta>0$ for some $\beta\notin F_n$. Then there exists
$\gamma\in \widehat{\mathbb{G}}$ such that $N_{\alpha,\gamma_1}^\gamma>0$ and $N_{\gamma,\gamma_2}^\beta>0$ for some $\beta\notin F_n$.

If $\gamma$ is not in $F_n$, then $\alpha$ is in $\{\eta\in F_n\,|N_{\eta,\gamma_1}^\zeta>0  {\rm\, for\, some\,}\zeta\notin F_n\}\subseteq \partial_{\gamma_1}(F_n)$.

If $\gamma$ is in $F_n$, then $\gamma$ is in $\{\eta\in F_n\,|N_{\eta,\gamma_2}^\zeta>0  {\rm\, for\, some\,}\zeta\notin F_n\}\subseteq\partial_{\gamma_2}(F_n)$.  We can define a map $\fai: \{\alpha\in F_n\,|N_{\alpha,\gamma_1\gamma_2}^\beta>0\, {\rm\, for\, some\,}\beta\notin F_n\}\to \{\eta\in F_n\,|N_{\eta,\gamma_2}^\zeta>0  {\rm\, for\, some\,}\zeta\notin F_n\}$ by $\fai(\alpha)=\gamma$ for some $\gamma$ with $N_{\gamma,\gamma_2}^\zeta>0$. By Lemma~\ref{lm: multiplicity}, we know that $\displaystyle\sum_{\alpha\in F_n,\,\fai(\alpha)=\gamma} d_\alpha^2\leq d_\gamma^2 d_{\gamma_1}^2$.

Therefore
\begin{align*}
&|\{\alpha\in F_n\,|N_{\alpha,\gamma_1\gamma_2}^\beta>0\, {\rm\, for\, some\,}\beta\notin F_n\}|_w \\
&\leq d_{\gamma_1}^2(|\{\eta\in F_n\,|N_{\eta,\gamma_1}^\zeta>0  {\rm\, for\, some\,}\zeta\notin F_n\}|_w+ |\{\eta\in F_n\,|N_{\eta,\gamma_2}^\zeta>0  {\rm\, for\, some\,}\zeta\notin F_n\}|_w).
\end{align*}

Moreover $$\{\alpha\notin F_n\,|N_{\alpha,\gamma_1\gamma_2}^\beta>0\, {\rm\, for\, some\,}\beta\in F_n\}=\{\alpha\notin F_n\,|N_{\alpha\gamma_1,\gamma_2}^\beta>0\, {\rm\, for\, some\,}\beta\in F_n\}.$$

As before if $\alpha$ is  not in $F_n$ with $N_{\alpha,\gamma_1\gamma_2}^\beta>0$ for some $\beta\in F_n$, then there exists
$\gamma\in \widehat{\mathbb{G}}$ such that $N_{\alpha,\gamma_1}^\gamma>0$ and $N_{\gamma,\gamma_2}^\beta>0$ for some $\beta\in F_n$.

If $\gamma$ is not in $F_n$, then $\gamma$ is in $\{\eta\notin F_n\,|N_{\eta,\gamma_2}^\zeta>0  {\rm\, for\, some\,}\zeta\in F_n\}\subseteq\partial_{\gamma_2}(F_n)$. We can define a map $\psi: \{\alpha\notin F_n\,|N_{\alpha,\gamma_1\gamma_2}^\beta>0\, {\rm\, for\, some\,}\beta\in F_n\}\to \{\eta\notin F_n\,|N_{\eta,\gamma_2}^\zeta>0  {\rm\, for\, some\,}\zeta\in F_n\}$ by $\psi(\alpha)=\gamma$ with  $N_{\gamma,\gamma_2}^\zeta>0$. By Lemma~\ref{lm: multiplicity}, we have  $\displaystyle\sum_{\alpha\notin F_n,\,\psi(\alpha)=\gamma} d_\alpha^2\leq d_\gamma^2 d_{\gamma_2}^2$.

If $\gamma$ is in $F_n$, then $\alpha$ is in $\{\eta\notin F_n\,|N_{\eta,\gamma_1}^\zeta>0  {\rm\, for\, some\,}\zeta\in F_n\}\subseteq\partial_{\gamma_1}(F_n)$.

Hence
\begin{align*}
  &|\{\alpha\notin F_n\,|N_{\alpha,\gamma_1\gamma_2}^\beta>0\, {\rm\, for\, some\,}\beta\in F_n\}|_w\\
  & \leq d_{\gamma_2}^2( |\{\eta\notin F_n\,|N_{\eta,\gamma_1}^\zeta>0  {\rm\, for\, some\,}\zeta\in F_n\}|_w+|\{\eta\notin F_n\,|N_{\eta,\gamma_2}^\zeta>0  {\rm\, for\, some\,}\zeta\in F_n\}|_w).
\end{align*}

Therefore $|\partial_{\gamma_1\gamma_2} (F_n)|_w\leq \max{\{d_{\gamma_1}^2,d_{\gamma_2}^2\}}(|\partial_{\gamma_1}(F_n)|_w+ |\partial_{\gamma_2}(F_n)|_w)$. This proves (3).

\end{proof}

 Give a sequence $\Sigma=\{F_n\}$ of finite subsets in $\widehat{\mathbb{G}}$, we say a subset $\Lambda$ of $\widehat{\mathbb{G}}$ has positive upper density with respect to $\Sigma$ if $\displaystyle\limsup_{n\to\infty}\frac{|\Lambda\cap F_n|_w}{|F_n|_w}>0$.

 For $\alpha,\beta\in\widehat{\mathbb{G}}$, we say  a subset $\Lambda$ of $\widehat{\mathbb{G}}$ contains $\alpha^j\beta$ if every  $\gamma$ contained in $\alpha^j\beta$ is also in $\Lambda$.

Motivated by Theorem~\ref{APDG}, we give the following conjecture.

\begin{conjecture}
 Suppose  a subset $\Lambda$ of $\widehat{\mathbb{G}}$ has positive upper density with respect to $\Sigma$. Then for every $k>0$, there exists $\alpha$ in $\widehat{\mathbb{G}}_\Sigma$, $\beta\in\widehat{\mathbb{G}}$ and $n>0$ such that $\Lambda$ contains $\alpha^{jn}\beta$ for all $0\leq j\leq k-1$.
\end{conjecture}

\subsection{Discrete Quantum Group Orbits and a Mean Ergodic Theorem for Discrete Quantum Groups}\

Given a compact quantum group $\mathbb{G}$, the dual $\widehat{\mathbb{G}}$ is a discrete quantum group. Conversely for a discrete quantum group $\Gamma$, there is a compact quantum group $\mathbb{G}$ such that $\Gamma=\widehat{\mathbb{G}}$. With this in mind, when we talk about a discrete quantum group, we mean the dual of a compact quantum group~\cite{PodlesWoronowicz1990,Vandaele1996,MaesVanDaele1998,KustermansVaes1999,KustermansVaes2000,SoltanWoronowicz2007}.

Consider an action of a discrete quantum group $\widehat{\mathbb{G}}$ on a Hilbert space $H$, that is, a representation $\pi$ of $C(\mathbb{G})$ on $H$. We come up a definition of discrete quantum orbits in a Hilbert space. Various definitions of compact quantum group orbits already appear in~\cite{Sain2009, Huang2016-1,DKSS2016}.

\begin{definition}

For $x$ in $H$, we define the \textbf{ orbit} of $x$ as
$$\{y\in H\,|\,\exists\,\alpha,\beta\in\widehat{\mathbb{G}}\, {\rm \,such\,that}\,\frac{\pi(\chi(\alpha))}{d_\alpha}x=\frac{\pi(\chi(\beta))}{d_\beta}y\},$$ and denote it by ${\rm Orb}(x,\widehat{\mathbb{G}})$.
\end{definition}

 The following are true.
\begin{proposition}~[Properties of orbits]\

  \begin{enumerate}
    \item $x\in {\rm Orb}(x,\widehat{\mathbb{G}})$;
    \item If $y\in {\rm Orb}(x,\widehat{\mathbb{G}})$, then $x\in {\rm Orb}(y,\widehat{\mathbb{G}})$;
    \item $\{\frac{\pi(\chi(\alpha))}{d_\alpha}x\}_{\alpha\in\widehat{\mathbb{G}}}\subseteq {\rm Orb}(x,\widehat{\mathbb{G}})$.
  \end{enumerate}
\end{proposition}
\begin{proof}
  (1) and (2) are immediate from the definition of orbit.

  Note that $\frac{\pi(\chi(\alpha))}{d_\alpha}x=\frac{\pi(\chi(\gamma_0))}{d_{\gamma_0}}(\frac{\pi(\chi(\alpha))}{d_\alpha}x)$ for every $\alpha$ in $\widehat{\mathbb{G}}$. Here $\gamma_0=1$ stands for the trivial representation of $\mathbb{G}$. This proves (3).
\end{proof}


We say a vector $x$ in $H$  \textbf{is fixed by $\Sigma$}~(a sequence of finite subsets in $\widehat{\mathbb{G}}$) if
$$\pi(\chi(\alpha))x=d_\alpha x$$ for every $\alpha$ in $\cup F_n$.

It's easy to see that the set of vectors fixed by $\Sigma$ is a subspace of $H$. Denote it by $H_\Sigma$.

In this subsection we prove a mean ergodic theorem for discrete quantum groups, which is a generalization of the mean ergodic theorem for amenable discrete quantum groups~\cite{Huang2016-2}.

Let $\Sigma=\{F_n\}$ be a sequence of finite subsets in $\widehat{\mathbb{G}}$. For a representation $\pi:A=C(\mathbb{G})\to B(H)$, consider the sequence of bounded linear operators $\{\frac{1}{|F_n|_w}\sum_{\alpha\in F_n}d_\alpha\pi(\chi(\alpha))\}_{n=1}^\infty$ on $H$. Under  weak operator topology, the unit ball of $B(H)$ is compact and contains $\{\frac{1}{|F_n|_w}\sum_{\alpha\in F_n}d_\alpha\pi(\chi(\alpha))\}_{n=1}^\infty$. Hence there exist limit points for $\{\frac{1}{|F_n|_w}\sum_{\alpha\in F_n}d_\alpha\pi(\chi(\alpha))\}_{n=1}^\infty$.

\begin{theorem}
~\label{thm: met}
  \begin{enumerate}
    \item[(i)] Suppose $T$ is a  limit  of $\{\frac{1}{|F_n|_w}\sum_{\alpha\in F_n}d_\alpha\pi(\chi(\alpha))\}_{n=1}^\infty$ in $B(H)$. Then $TP=PT=P$ where $P$ is the orthogonal projection from $H$ onto $H_\Sigma:=\{x\in H|\,\pi(\chi(\alpha))x=d_\alpha x\, {\rm for\,all}\,\alpha\in\cup F_n\}$;
    \item[(ii)] If $y$ in $H$ belongs to ${\rm Orb}(x,\widehat{\mathbb{G}}_\Sigma)$, then
    $$\lim_{n\to\infty}\frac{1}{|F_n|_w}\sum_{\alpha\in F_n}d_\alpha\pi(\chi(\alpha))(x-y)=0.$$
  \end{enumerate}
\end{theorem}
\begin{proof}

(i) For every $x$ in $H_\Sigma$ and every $n$, we have $\frac{1}{|F_n|_w}\sum_{\alpha\in F_n}d_\alpha\pi(\chi(\alpha))x=x$. Hence $TP=P$.
Next we prove that $T^*P=P$ and this will finish the proof.

We need a lemma.
\begin{lemma}~\label{lm: bar}
  If $\pi(\chi(\alpha))x=d_\alpha x$, then $\pi(\chi(\bar{\alpha}))x=d_{\bar{\alpha}} x$.
\end{lemma}
\begin{proof}
Without loss of generality, we may assume that $x$ is a unit vector in $H$.

Then
\begin{align*}
 &0\leq \|\pi(\chi(\bar{\alpha}))x-d_{\bar{\alpha}} x\|^2\\
 &=\langle\pi(\chi(\bar{\alpha}))x,\pi(\chi(\bar{\alpha}))x\rangle-\langle\pi(\chi(\bar{\alpha}))x,d_{\bar{\alpha}} x\rangle-\langle d_{\bar{\alpha}} x,\pi(\chi(\bar{\alpha}))x\rangle+\langle d_{\bar{\alpha}} x, d_{\bar{\alpha}} x\rangle \\
  \tag{$\pi(\chi(\bar{\alpha}))x=d_{\bar{\alpha}} x$ implies that $\langle\pi(\chi(\bar{\alpha}))x,d_{\bar{\alpha}} x\rangle=\langle d_{\bar{\alpha}} x,\pi(\chi(\bar{\alpha}))x\rangle=d_\alpha^2=d_{\bar{\alpha}}^2$.} \\
 &=\langle\pi(\chi(\bar{\alpha}))x,\pi(\chi(\bar{\alpha}))x\rangle-d_{\bar{\alpha}}^2
 \tag{$\|\chi(\bar{\alpha})\|\leq d_{\bar{\alpha}}$.}
 \leq 0.
\end{align*}

Hence $\pi(\chi(\bar{\alpha}))x=d_{\bar{\alpha}} x$.
\end{proof}

Note that $(\frac{1}{|F_n|_w}\sum_{\alpha\in F_n}d_\alpha\pi(\chi(\alpha)))^*=\frac{1}{|F_n|_w}\sum_{\alpha\in F_n}d_\alpha\pi(\chi(\bar{\alpha}))$. By Lemma~\ref{lm: bar}, we have $\frac{1}{|F_n|_w}\sum_{\alpha\in F_n}d_\alpha\pi(\chi(\bar{\alpha}))x=x$ for every $x$ in $H_\Sigma$. This implies that $T^*P=P$.

(ii) We first prove that
\begin{equation}\label{eq: met}
\lim_{n\to\infty}\frac{1}{|F_n|_w}\sum_{\alpha\in F_n}d_\alpha\pi(\chi(\alpha))(\pi(\chi(\gamma))y-d_\gamma y)=0
\end{equation}
 for all $y$ in $H$ and $\gamma$ in $\widehat{\mathbb{G}}_\Sigma$. The proof is a simplified version of the proof of~\cite[Thm. 3.1]{Huang2016-2}

For every $y\in H$ and $\gamma\in \widehat{G}$, we have

\begin{align*}
&\sum_{\alpha\in F_n}d_\alpha\pi(\chi(\alpha))(\pi(\chi(\gamma))y-d_\gamma y) \\
\tag{$\chi(\alpha)\chi(\gamma)=\chi(\alpha\gamma)$ and $d_\gamma=d_{\bar{\gamma}}$.}
&=\sum_{\alpha\in F_n}d_\alpha\pi(\chi(\alpha\gamma))y-\sum_{\alpha\in F_n}d_\alpha d_{\bar{\gamma}}\pi(\chi(\alpha))y   \\
&=\sum_{\alpha\in F_n}d_\alpha\sum_{\beta\in\widehat{\mathbb{G}}}N_{\alpha,\gamma}^\beta\pi(\chi(\beta))y-
\sum_{\alpha\in F_n}\sum_{\beta\in\widehat{\mathbb{G}}}N_{\alpha,\bar{\gamma}}^\beta d_\beta\pi(\chi(\alpha))y.
\end{align*}

Note that
$$\sum_{\alpha\in F_n}d_\alpha\sum_{\beta\in\widehat{\mathbb{G}}}N_{\alpha,\gamma}^\beta\pi(\chi(\beta))y=\sum_{\alpha\in F_n}d_\alpha(\sum_{\beta\in F_n}+\sum_{\beta\notin F_n})N_{\alpha,\gamma}^\beta\pi(\chi(\beta))y.$$
 Moreover
 \begin{align*}
 &\sum_{\alpha\in F_n}\sum_{\beta\in\widehat{\mathbb{G}}}N_{\alpha,\bar{\gamma}}^\beta d_\beta\pi(\chi(\alpha))y  \\
 &=\sum_{\alpha\in F_n}\sum_{\beta\in F_n}N_{\alpha,\bar{\gamma}}^\beta d_\beta\pi(\chi(\alpha))y+\sum_{\alpha\in F_n}\sum_{\beta\notin F_n}N_{\alpha,\bar{\gamma}}^\beta d_\beta\pi(\chi(\alpha))y \\
\tag{$N_{\alpha,\bar{\gamma}}^\beta=N_{\beta,\gamma}^\alpha$.}
&=\sum_{\alpha\in F_n}\sum_{\beta\in F_n}N_{\beta,\gamma}^\alpha d_\beta\pi(\chi(\alpha))y+\sum_{\alpha\in F_n}\sum_{\beta\notin F_n}N_{\alpha,\bar{\gamma}}^\beta d_\beta\pi(\chi(\alpha))y \\
&=\sum_{\alpha\in F_n}\sum_{\beta\in F_n}N_{\alpha,\gamma}^\beta d_\alpha\pi(\chi(\beta))y+\sum_{\alpha\in F_n}\sum_{\beta\notin F_n}N_{\alpha,\bar{\gamma}}^\beta d_\beta\pi(\chi(\alpha))y.
 \end{align*}
Hence
\begin{align*}
&\frac{1}{|F_n|_w}\sum_{\alpha\in F_n}d_\alpha\pi(\chi(\alpha))(\pi(\chi(\gamma))y-d_\gamma y) \\
&=\frac{1}{|F_n|_w}\sum_{\alpha\in F_n}\sum_{\beta\notin F_n}d_\alpha N_{\alpha,\gamma}^\beta\pi(\chi(\beta))y-\frac{1}{|F_n|_w}\sum_{\alpha\in F_n}\sum_{\beta\notin F_n}N_{\alpha,\bar{\gamma}}^\beta d_\beta\pi(\chi(\alpha))y.
\end{align*}

Note that
\begin{align*}
 &\|\frac{1}{|F_n|_w}\sum_{\alpha\in F_n}\sum_{\beta\notin F_n}d_\alpha N_{\alpha,\gamma}^\beta\pi(\chi(\beta))y\|   \\
 \tag{$\|\chi(\beta)\|\leq d_\beta$.}
 &\leq \frac{1}{|F_n|_w}\sum_{\alpha\in F_n}\sum_{\beta\notin F_n}d_\alpha N_{\alpha,\gamma}^\beta d_\beta\|y\|   \\
 &\leq \frac{1}{|F_n|_w}\sum_{\alpha\in \partial_\gamma F_n}d_\alpha^2 d_\gamma\|y\|\to 0
\end{align*}
as $n\to\infty$ since $\gamma$ is in $\widehat{\mathbb{G}}_\Sigma$.

Also
\begin{align*}
&\|\frac{1}{|F_n|_w}\sum_{\alpha\in F_n}\sum_{\beta\notin F_n}N_{\alpha,\bar{\gamma}}^\beta d_\beta\pi(\chi(\alpha))y\|   \\
&\leq \frac{1}{|F_n|_w}\sum_{\alpha\in F_n}\sum_{\beta\notin F_n}N_{\alpha,\bar{\gamma}}^\beta d_\beta d_\alpha\|y\|  \\
&\leq \frac{1}{|F_n|_w}\sum_{\alpha\in\partial_{\bar{\gamma}}F_n}d_\alpha^2d_{\bar{\gamma}}\|y\|\to 0
\end{align*}

as $n\to\infty$ since Theorem~\ref{thm: gp} guarantees that when $\gamma$ is in $\widehat{\mathbb{G}}_\Sigma$,  so is $\bar{\gamma}$.

Therefore
$$\lim_{n\to\infty}\frac{1}{|F_n|_w}\sum_{\alpha\in F_n}d_\alpha\pi(\chi(\alpha))(\pi(\chi(\gamma))y-d_\gamma y)=0.$$

If $y$ is in ${\rm Orb}(x,\widehat{\mathbb{G}}_\Sigma)$, then  there exist $\beta$ and $\gamma$ in $\widehat{\mathbb{G}}_\Sigma$ such that $\frac{\pi(\chi(\beta))}{d_\beta}x=\frac{\pi(\chi(\gamma))}{d_\gamma}y$. From Equation~\ref{eq: met}, we have
\begin{align*}
&\lim_{n\to\infty}\frac{1}{|F_n|_w}\sum_{\alpha\in F_n}d_\alpha\pi(\chi(\alpha))(x-y)  \\
\tag{$\frac{\pi(\chi(\beta))}{d_\beta}x=\frac{\pi(\chi(\gamma))}{d_\gamma}y$}
&=\lim_{n\to\infty}\frac{1}{|F_n|_w}\sum_{\alpha\in F_n}d_\alpha\pi(\chi(\alpha))[(x-\frac{\pi(\chi(\beta))}{d_\beta}x)+(\frac{\pi(\chi(\gamma))}{d_\gamma}y-y)]=0.
\end{align*}

\end{proof}

\begin{remark}
Theorem~\ref{thm: met}(ii) says that if  $x$ and $y$ are in the same orbit of $\widehat{\mathbb{G}}_\Sigma$, then averages of $x$ and $y$ along $\Sigma$ coincide.
\end{remark}



\section*{Acknowledgements}
I thank Hanfeng Li for his illuminating comments. I got familiar with Furstenberg correspondence principle in a 2012 graduate student seminar  in SUNY at Buffalo organized by  Bingbing Liang, Yongle Jiang, Yongxiao Lin and myself. I thank them  for their kind feedback.   The latest version of the paper was carried out during a visit to the Research Center for Operator Algebras in East China Normal University in April 2017. I thank Huaxin Lin for his hospitality and Qin Wang for helpful discussions.


\begin{thebibliography}{999}

\small

\bibitem[BBF10]{BBF2010}
M. Beiglb\"ock, V. Bergelson and A. Fish.  Sunset phenomenon in countable amenable groups.  {\it Adv. Math.}  \textbf{ 223} (2010), no. 2, 416--432.



\bibitem[BL96]{BL1996}
V. Bergelson and A. Leibman.  Polynomial extensions of van der Waerden's and Szemer\'edi's theorems. {\it J. Amer. Math. Soc.} \textbf{ 9} (1996), 725--753.

\bibitem[BMT01]{BMT2001}
E. B\'edos, G. J. Murphy and L. Tuset. Co-amenability of compact quantum groups. {\it J. Geom. Phys.}  \textbf{ 40}  (2001),  no. 2, 130--153.

\bibitem[BS93]{BaajSkandalis1993}
S. Baaj and G. Skandalis. Unitaires multiplicatifs et dualit\'e pour les produits crois\'es de $C^*$-alg\`{e}bres. {\it Ann. Sci. \'Ecole Norm. Sup.} (4) \textbf{ 26} (1993), no. 4, 425--488.

\bibitem[DKSS16]{DKSS2016}
K. De Commer, P. Kasprzak, A.Skalski and P. M.Soltan. Quantum actions on discrete quantum spaces and a generalization of Clifford's theory of representations.  arXiv:1611.10341v1.

\bibitem[ET36]{ET1936}
P. Erd\"os and P. Tur\'an.  On some sequences of integers. {\it J.  Lond. Math. Soc.} \textbf{11} (1936), no. 4, 261--264.



\bibitem[Fur77]{Furstenberg1977}
 H. Furstenberg.  Ergodic behavior of diagonal measures and a theorem of Szemer\'edi on arithmetic progressions.  {\it J. Anal. Math.}  \textbf{ 31}, (1977), 204--256.

\bibitem[Fur81]{Furstenberg1981}
H. Furstenberg. {\it Recurrence in Ergodic theory and Combinatorial Number Theory}. Princeton University Press, Princeton NJ, 1981.

\bibitem[FK78]{FK1978}
H. Furstenberg and Y. Katznelson. An ergodic Szemer\'edi theorem for commuting transformations. {\it J. Anal. Math.} \textbf{ 34} (1978), 275--291.

\bibitem[FK91]{FK1991}
H. Furstenberg and Y. Katznelson.  A density version of the Hales-Jewett theorem. {\it J.  Anal. Math.} \textbf{ 57} (1991), 64--119.

\bibitem[FKO82]{FKO1982}
H. Furstenberg, Y.  Katznelson and D. S. Ornstein.  The ergodic theoretical proof of Szemer\'edi's theorem. {\it Bull. Amer. Math. Soc.} \textbf{ 7} (1982), no. 3, 527--552.


\bibitem[GT08]{GreenTao2008}
B. Green and T. Tao.  The primes contain arbitrarily long arithmetic progressions. {\it Ann. of Math.} \textbf{ 167} (2008), no. 2, 481--547.

\bibitem[Hua16-1]{Huang2016-1}
H. Huang. Invariant subsets under compact quantum group actions. {\it J. Noncommut. Geom.} \textbf{ 10} (2) (2016), 447--469.

\bibitem[Hua16-2]{Huang2016-2}
H. Huang. Mean ergodic theorem for amenable discrete quantum groups and a Wiener-type theorem for compact metrizable groups. {\it Anal. PDE} \textbf{ 9} (2016), no. 4, 893--906.


\bibitem[Kye08]{Kyed2008}
D. Kyed.  $L^2$-Betti numbers of coamenable quantum groups. {\it M\"{u}nster J. Math. } \textbf{ 1} (2008), 143–-179.


\bibitem[KV99]{KustermansVaes1999}
J. Kustermans and S.  Vaes. A simple definition for locally compact quantum groups. {\it C. R. Acad. Sci. Paris S\'er. I Math.} \textbf{ 328} (1999), no. 10, 871--876.

\bibitem[KV00]{KustermansVaes2000}
J. Kustermans and S.  Vaes. Locally compact quantum groups. {\it Ann. Sci. \'Ecole Norm. Sup.} (4) \textbf{ 33} (2000), no. 6, 837--934.


\bibitem[MvD98]{MaesVanDaele1998}
M. Maes and A. Van Daele.  Notes on compact quantum groups. {\it Nieuw Arch. Wisk.} (4) \textbf{16} (1998), no. 1-2, 73--112.

\bibitem[PW90]{PodlesWoronowicz1990}
P. Podle\'{s}, P and S. L. Woronowicz. Quantum deformation of Lorentz group. {\it Comm. Math. Phys.} \textbf{ 130} (1990), no. 2, 381--431.

\bibitem[Sai09]{Sain2009}
J. N. Sain. Berezin quantization from ergodic actions of compact quantum groups, and quantum Gromov-Hausdorff distance. Thesis (Ph.D.)–University of California, Berkeley. 2009.

\bibitem[Sze69]{Szemeredi1969}
E. Szemer\'edi. On sets of integers containing no four elements in arithmetic progression. {\it Acta Math. Acad. Sci. Hungar.} \textbf{ 20} (1969), 89--104.


\bibitem[Sze75]{Szemeredi1975}
E. Szemer\'edi. On sets of integers containing no $k$ elements in arithmetic progression. {\it Acta Arith.} \textbf{ 27}, (1975), 199--245.

\bibitem[SW07]{SoltanWoronowicz2007}
P. M. So{\l}tan and S. L. Woronowicz. From multiplicative unitaries to quantum groups. II. {\it J. Funct. Anal.} {\bf 252} (2007), no. 1, 42–-67.

\bibitem[Tao07]{Tao2007}
T. Tao.  What is good mathematics? {\it Bull. Amer. Math. Soc. } \textbf{ 44} (2007), no. 4, 623--634.

\bibitem[vanD96]{Vandaele1996}
A. Van Daele. Discrete quantum groups. {\it J. Algebra} \textbf{ 180} (1996), no. 2, 431--444.

\bibitem[vand27]{vanderWaerden1927}
B. L. van der Waerden. Beweis einer Baudetschen Vermutung. {\it Nieuw. Arch. Wisk.} \textbf{15} (1927), 212--216.

\bibitem[Wal82]{Walters1982}
P. Walters. {\it An Introduction to Ergodic Theory.} Graduate Texts in Mathematics, \textbf{ 79}. Springer-Verlag, New York-Berlin, 1982.

\bibitem[Wor87]{Woronowicz1987}
S. L. Woronowicz. Compact matrix pseudogroups. {\it Comm. Math. Phys.} \textbf{ 111} (1987), no. 4, 613--665.

\bibitem[Wor98]{Woronowicz1998}
S. L. Woronowicz. Compact quantum groups. {\it Sym\'etries Quantiques (Les Houches, 1995)},  845--884, North-Holland, Amsterdam, 1998.

\end{thebibliography}
\end{document}